\documentclass{amsart}
\usepackage[english]{babel}

\usepackage{amsmath}

\usepackage{graphicx}
\usepackage{float}

\usepackage{amsfonts,amssymb,amsthm}

\usepackage{yhmath}

\usepackage{xcolor}

\usepackage{caption}

\usepackage{enumitem}

\usepackage{setspace}
\doublespace

\theoremstyle{plain}
\newtheorem{definition}{Definition}[section]
\newtheorem{proposition}{Proposition}[section]

\newtheorem{remark}{Remark}[section]
\newtheorem*{customconjecture}{Conjecture}

\newtheorem*{customtheorem}{Theorem}
\newtheorem{theorem}{Theorem}[section]

\usepackage{csquotes}
\PassOptionsToPackage{hyphens}{url}
\usepackage{url}
\usepackage{cleveref}

\makeatletter
\renewcommand*\@maketitle{%
  \normalfont\normalsize
  \@adminfootnotes
  \@mkboth{\@nx\shortauthors}{\@nx\shorttitle}%
  \global\topskip42\p@\relax % 5.5pc   "   "   "     "     "
  \@settitle
  \ifx\@empty\authors \else \@setauthors \fi
  \ifx\@empty\@date \else {\vskip 1em \vtop{\centering\large\@date\@@par}}\fi% MY CHANGE
  \ifx\@empty\@dedicatory
  \else
    \baselineskip18\p@
    \vtop{\centering{\footnotesize\itshape\@dedicatory\@@par}%
      \global\dimen@i\prevdepth}\prevdepth\dimen@i
  \fi
  \@setabstract
  \normalsize
  \if@titlepage
    \newpage
  \else
    \dimen@34\p@ \advance\dimen@-\baselineskip
    \vskip\dimen@\relax
  \fi
} % end \@maketitle
\renewcommand*\@adminfootnotes{%
  \let\@makefnmark\relax  \let\@thefnmark\relax
  \ifx\@empty\@subjclass\else \@footnotetext{\@setsubjclass}\fi
  \ifx\@empty\@keywords\else \@footnotetext{\@setkeywords}\fi
  \ifx\@empty\thankses\else \@footnotetext{%
    \def\par{\let\par\@par}\@setthanks}%
  \fi
}
\makeatother

\makeatletter
\renewcommand\subsection{\@startsection{subsection}{2}%
  \z@{.5\linespacing\@plus.7\linespacing}{-.5em}%
  {\normalfont\scshape}}
\makeatother

\title{On smooth gaps between primes using the Maynard-Tao sieve}
\author{Carol Wu}
\address{Port Moody Secondary School,
300 Albert St, Port Moody, BC V3H 2M5}
\email{carolkl.wu@gmail.com}
\date{\today}

\begin{document}

\maketitle

\begin{abstract}
    In 1999, Balog, Br\"udern, and Wooley \cite{MR1750403} showed there are infinitely many prime gaps $p-q$ that are $(\log p)^{\frac{3}{4}}$-smooth, and infinitely many consecutive prime gaps that are $(\log p)^\frac{7}{8}$-smooth. Advancements made since then by Zhang \cite{MR3171761}, Maynard \cite{MR3272929}, and Polymath8b \cite{MR3373710} towards resolving the twin prime conjecture have given us the tools to lower the bounds made by Balog, Br\"udern, and Wooley to 47. Moreover, we can show there are infinitely many $m$-tuples of primes whose gaps are all $y_m$-smooth for a calculable prime $y_m$.
\end{abstract}

\section{\centering Introduction}
    The twin prime conjecture posits there are infinitely many pairs of primes $p$, $q$ such that $p-q = 2$. Whilst unresolved, a significant result from Polymath8b \cite[Theorem~1.4]{MR3373710}, building upon the work of Goldston, Pintz, and Y\i ld\i r\i m \cite{MR2552109}; Zhang \cite{MR3171761}; and Maynard \cite{MR3272929}, asserts there are infinitely many pairs of primes $p > q$ such that $p-q \leq 246$ (a result we will elaborate on later).

    An alternative way to weaken the twin prime conjecture is to assert there are infinitely many pairs of primes $p$, $q$ such that $p-q = 2^n$, or, to restate, whose difference is 2-smooth.\footnote{An integer is $y$-smooth if its largest prime factor is less than or equal to $y$.} However, a proof of this conjecture, too, is beyond current techniques. So is the conjecture stating there are infinitely many pairs of primes $p$, $q$ whose gap is $3$-smooth, or $5$-smooth. 
    
    We define $\mathcal{S}(y)$ to be the set of all $y$-smooth integers. In 1999, Balog, Br\"udern, and Wooley \cite{MR1750403} showed there are infinitely many pairs of primes $p$, $q \asymp x$ such that $p-q \in \mathcal{S}\left(\left(\log x\right)^\frac{3}{4}\right)$, and infinitely many \textit{consecutive} primes $p$, $q \asymp x$ such that $p-q \in \mathcal{S}\left(\left(\log x\right)^\frac{7}{8}\right)$. Moreover, it follows immediately from the aforementioned results given by Polymath8b there are infinitely many pairs of primes $p$, $q$ with gap $p-q \in \mathcal{S}\left(241\right)$. In this paper, we show:

    \begin{theorem} \label{thm: inf 47-smooth pairs}
        There exists infinitely many primes $p$, $q$ such that $p-q \in \mathcal{S}(47)$.
    \end{theorem}
    
    To this effect, we shall use a celebrated theorem due to Maynard \cite{MR3272929} (and independently Tao \cite{tao2013boundedintervals}), for which we first require some definitions to state precisely:
    
    \begin{definition}
        A $k$-tuple of integers $H_k$ is \textbf{admissible} if, for every prime $p$, the elements of $H_k$ do not cover all congruence classes $\bmod$ $p$.
    \end{definition}

    \begin{definition}
        The \textbf{diameter} of a $k$-tuple is the difference between its largest and smallest element.
    \end{definition}

    The following version of the Maynard(-Tao) Theorem was formulated by Banks, Freiberg, and Turnage-Butterbaugh \cite{MR3316460}:
    
    \begin{customtheorem} [Maynard(-Tao) Theorem] \label{thm: Maynard(-Tao)}
        For any positive integer $m \geq 2$, there exists some $k_m \in \mathbb{N}$ such that, for any admissible $k$-tuple of integers $H_{k}$ where $k \geq k_m$, there exists an infinite number of integers $n$ such that $n +h$ is prime for at least $m$ elements $h \in H_k$. 
    \end{customtheorem}

    The $m=2$ case was proved by Zhang in 2013, with $k_2 = 3.5 \cdot 10^6$ \cite[Theorem~1]{MR3171761}. Polymath8a \cite{MR3294387} refined his method such that one could take $k_2 = 632$. Maynard \cite[Proposition~4.3]{MR3272929} discovered a simpler method that lowered $k_2$ to 105 and showed $k_m$ to be finite for all $m$, giving $k_m < cm^2e^{4m}$ for some absolute constant $c$. Subsequent optimizations made by Polymath8b \cite[Theorem~3.2]{MR3373710} and Stadlmann \cite[Corollary~1]{stadlmann2023primes} give us the current lowest known values of $k_m$ recorded in \Cref{tab: k_m values}.

        \begin{table}[] 
        \centering
        \begin{tabular}{c|c}
             $m$ & $k_m$ \\
             \hline
             2 & 50 \\
             3 & 35265 \\
             4 & 1624545 \\
             5 & 73807570 \\
             6 & 3340375663 \\
        \end{tabular}
        \caption{Lowest known values of $k_m$}
        \label{tab: k_m values}
    \end{table}
    
    \begin{remark}\label{rem: stadlmann k_m}
        For $m > 6$, Stadlmann \cite[Theorem~2]{stadlmann2023primes} showed $k_m < ce^{3.8075m}$ for some absolute constant $c$.
    \end{remark}

    \begin{remark}\label{fac: engelsma diameter admissible}
        Computations by Engelsma in an unpublished work \cite[Table~5]{engelsma2009} show that an admissible $50$-tuple of integers with diameter 246 exists, and that this diameter is minimal over all admissible $50$-tuples.
    \end{remark}
    
    \begin{remark}
        An immediate consequence of the above fact, and the Maynard(-Tao) Theorem, is the existence of infinitely many pairs of primes $p > q$ such that $p-q\leq 246$.
    \end{remark}

    \begin{remark} \label{rem: elliott-halberstam}
        Maynard \cite{MR3272929} showed $k_2$ can be lowered to 5 assuming the Elliott-Halberstam Conjecture. 
    \end{remark}
    
    The twin prime conjecture belongs to a larger group of problems on the distribution of patterns of small prime gaps. Such patterns can be represented by a $k$-tuple of integers $H_k$. In generalizing the twin prime conjecture, one can ask: for how many integers $n$ is $n + h$ prime for all $h \in H_k$? The case $H_2 = ( 0, 2 )$ corresponds to the twin prime conjecture. In 1923, Hardy and Littlewood \cite{MR1555183}, through probabilistic reasoning, made the following conjecture (the formulation we present here follows Tao \cite{tao2013primegapshistory}):

    \begin{customconjecture} [Hardy-Littlewood $k$-tuples Conjecture]
        
         Given an admissible $k$-tuple $H_k$, we define 
         $$\mathfrak{G} = \mathfrak{G}\left(H_k\right) := \prod\limits_{p \in \mathcal{P}} \frac{1-\frac{v_p}{p}}{\left(1-\frac{1}{p}\right)^k}$$  
         where $\mathcal{P}$ is the set of all primes and $v_p$ is the number of congruence classes $\bmod$ $p$ covered by the elements in $H_k$. 
         
         Then, the number of natural numbers $ n < x $ such that $ n + H_k$ consists entirely of primes is asymptotic to $\mathfrak{G}\frac{x}{(\log x)^k}$.
    \end{customconjecture}

    This conjecture is generally believed to be true and appears consistent with experimental data. Moreover, the Maynard(-Tao) Theorem is a significant stride towards it. We will be applying the Maynard(-Tao) Theorem to the question of smooth gaps between primes, with the following results: 
    
    \begin{theorem} \label{thm: inf p-smooth primes}
        Let $y_m$ represent the largest prime less than or equal to $k_m$, where $k_m$ is as in \Cref{tab: k_m values}, \Cref{rem: stadlmann k_m}, and \Cref{rem: elliott-halberstam}. Then, there exists infinitely many $m$-tuples of primes $P_m = \left(p_1\text{, } p_2\text{, } p_3\text{, } ... \text{, } p_m\right)$ such that  $p_i - p_j \in \mathcal{S}(y_m)$ for all $1 \leq j < i \leq m$.
    \end{theorem}

    Note that \Cref{thm: inf 47-smooth pairs} is simply a special case of \Cref{thm: inf p-smooth primes}, where $m = 2$, $k_m = 50$, and $y_m = 47$. 

    \begin{remark}
        Improvements on the bound on $k_m$ would improve upon \Cref{thm: inf p-smooth primes}, lowering the bound $y_m$.
    \end{remark}

    For example:
    \begin{theorem} \label{thm: elliott-halberstam inf 5-smooth primes}
        Assuming the Elliott-Halberstam Conjecture, there exists infinitely many pairs of primes $p$, $q$ such that $p-q \in \mathcal{S}(5)$.
    \end{theorem}

    \begin{proof}
        Follows directly from \Cref{rem: elliott-halberstam} and \Cref{thm: inf p-smooth primes}.
    \end{proof}

    \begin{definition}
        A tuple $H$ is \textbf{difference y-smooth} if $h_i - h_j \in \mathcal{S}(y)$ for every pair of elements $h_i, h_j \in H$.
    \end{definition} 

    \begin{theorem} \label{thm: admissible and y-k smooth sets}
        For any integer $k$, there exists a $k$-tuple of integers $H_k$ such that $H_k$ is admissible and $H_k$ is difference $k$-smooth.
    \end{theorem}
        
        \Cref{thm: inf p-smooth primes} follows from \Cref{thm: admissible and y-k smooth sets}, which we will show in \Cref{sec: proof inf p-smooth primes}.
        
    \begin{remark}
        Balog, Br\"udern, and Wooley \cite{MR1750403} distinguished between smooth prime gaps and consecutive smooth gaps. However, Banks, Freiberg, and Turnage-Butterbaugh \cite[Theorem~1]{MR3316460} established the admissible $k$-tuples in the Maynard(-Tao) Theorem also infinitely often represent consecutive prime gaps. Hence, there is no need to distinguish between consecutive and non-consecutive smooth prime gaps. 
    \end{remark}
\subsection*{\centering Acknowledgements}
    I would like to express my sincere gratitude to Dr. Anurag Sahay for his invaluable guidance and mentorship throughout the research process, and to Dr. Trevor Wooley for useful discussions and feedback.
    
\section{Proof of Theorem 1.2} \label{sec: proof inf p-smooth primes}
    
    \begin{proof} [Proof of \Cref{thm: inf p-smooth primes}]
        
        Assuming \Cref{thm: admissible and y-k smooth sets}, let $H$ be a $k_m$-tuple that is admissible and difference $k_m$-smooth. 
        
        Since $H$ is admissible and contains at least $k_m$ elements, then, by the Maynard(-Tao) Theorem, there are infinitely many $n \in \mathbb{N}$ such that $n+h$ is prime for at least $m$ elements $h \in H$. 

        Let $P_m (n) = (p_1 \text{, } p_2 \text{, } p_3 \text{, } ... \text{, } p_m)$ be such an $m$-tuple of primes.

        For any two elements of $p_i$, $p_j \in P_m(n)$, $p_i - p_j = h_s - h_t$, where $h_s, h_t \in H$. Recall $H$ is difference $k_m$-smooth, so, $p_i - p_j = h_s-h_t \in \mathcal{S}(y_m)$. 
        
        Therefore, there are infinitely many $m$-tuples of primes $P_m (n)$ with gaps $p_i - p_j \in \mathcal{S}(k_m) = \mathcal{S}(y_m)$ for all pairs of elements $p_i$, $p_j \in P_m(n)$.
    \end{proof}

\section{\centering Proof of Theorem 1.4} 
    \begin{proof} [Proof of \Cref{thm: admissible and y-k smooth sets}]
        Let $\omega$ be the product of all positive primes up to $k$ inclusive.
    
        Consider the arithmetic progression $H_k = (0, \omega, 2\omega, 3\omega, ..., (k-1)\omega)$.

        $H_k$ has all the desired properties: 
        
        \begin{enumerate}
            \item There are $k$ elements.
            \item For every prime $p \leq k$, $p \mid \omega$. So, every element of $H_k$ is $0 \bmod p$. For every prime $p > k$, there are at least $k + 1$ congruence classes mod $p$, but only $k$ elements in $H_k$, and therefore, the elements of $H_k$ cannot cover every congruence class mod $p$. Hence, $H_k$ is admissible. 
            \item The difference between any two elements of $H_k$ is $a\omega$ for some integer $0 < a < k$. We define $y$ to be the largest prime not exceeding $k$. Then, both $a$ and $\omega \in \mathcal{S}(y)$, so $a\omega \in \mathcal{S}(y)$. Therefore, $H_k$ is difference $y$-smooth.
        \end{enumerate} 
    \end{proof}
\section{\centering Optimality of $y_m$}
    A natural question arises from the above formulation: Is the bound $y_m$ optimal, or can it be lowered? 
    \begin{proposition} \label{prop: y_m limit}
        Let $z_k$ represent the largest prime less than or equal to $k$. Then, $H_k$ is admissible implies $H_k$ is \textit{not} difference $\ell$-smooth for any $\ell < z_k$.
    \end{proposition}

    \begin{proof}
        Consider the elements of $H_k \bmod z_k$. 

        Because $H_k$ is admissible, its elements cannot cover all congruence classes mod $z_k$. At most, its $k$ elements can cover $z_k-1$ congruence classes. But $k > z_k - 1$, so by the pigeonhole principle, there exists some pair of elements $h_i, h_j \in H_k$ in the same congruence class mod $z_k$. 

        Then, $h_i - h_j \equiv 0 \bmod z_k$, so $z_k \mid h_i - h_j$. But $z_k > \ell$ and $z_k$ is prime, so $z_k$ is a witness for $H_k$ not being difference $\ell$-smooth for any $\ell < z_k$. 
    \end{proof}

    \begin{remark}
        As a consequence of \Cref{prop: y_m limit}, improving \Cref{thm: inf p-smooth primes} requires a different method, one avoiding the Maynard(-Tao) sieve.
    \end{remark}

\bibliographystyle{plain}
\bibliography{sources.bib} 

\end{document}